\documentclass{amsart}

\usepackage{amssymb}
\usepackage{amsmath}
\usepackage{amsthm}
\usepackage{amsbsy}
\usepackage{bm}

\theoremstyle{plain}
\newtheorem{theorem}{Theorem}[section]
\newtheorem{lemma}[theorem]{Lemma}

\theoremstyle{definition}

\theoremstyle{remark}

\newcommand{\Z}{\mathbb{Z}}

\newcommand{\co}{\colon\thinspace}
\newcommand{\del}{\partial}

\newcommand{\hpi}{\hat\pi}

\begin{document} 

\title{The Goldman bracket characterizes homeomorphisms}

\author{Siddhartha Gadgil}

\address{	Department of Mathematics,\\
		Indian Institute of Science,\\
		Bangalore 560012, India}

\email{gadgil@math.iisc.ernet.in}


\date{\today}

\begin{abstract}
We show that a homotopy equivalence between compact, connected, oriented surfaces with non-empty boundary is homotopic to a homeomorphism if and only if it commutes with the Goldman bracket.
\\

Nous montrons que d'une \'equivalence d'homotopie entre les compacte, con-
connectés, les surfaces orientées avec bord non vide est homotope \`a un hom\'eo-
morphisme si et seulement si elle commute avec le  Goldman bracket.
\end{abstract}

\maketitle

\section{Introduction}

Given a homotopy equivalence between manifolds of the same dimension, a fundamental question in topology is whether it is homotopic to a homeomorphism. Basic examples of exotic homotopy equivalences -- ones that are not  homotopic to the homeomorphisms, are maps between surfaces with boundary. For example, the three-holed sphere and the one-holed torus are homotopy equivalent but not homeomorphic. 

We assume throughout that all surfaces we consider are oriented. In this note, we show that there is a simple and natural characterisation of when a homotopy equivalence $f\co \Sigma_1\to\Sigma_2$ between compact surfaces is homotopic to a homeomorphism in terms of the \emph{Goldman bracket}, which is a Lie Algebra structure associated to a surface. More precisely, if $\hpi(\Sigma)$ denotes the set of free homotopy classes of closed curves in $\Sigma$, then the Goldman bracket (whose definition we recall below) is a bilinear map
$$[\cdot,\cdot]\co \Z[\hpi(\Sigma)]\times\Z[\hpi(\Sigma)]\to\Z[\hpi(\Sigma)],$$
which is skew-symmetric and satisfies the Jacobi identity.

\begin{theorem}\label{main}
A homotopy equivalence $f\co \Sigma_1\to \Sigma_2$ between compact, connected, oriented surfaces with non-empty boundary is homotopic to a homeomorphism if and only if it commutes with the Goldman bracket, i.e., for all $x,y\in\Z[\hpi(\Sigma_1)]$, we have,
\begin{equation}\label{preserve}
[f_*(x),f_*(y)]=f_*([x,y]),
\end{equation}
where $f_*\co \Z[\hpi(\Sigma_1)]\to\Z[\hpi(\Sigma_1)]$ is the function induced by $f$.
\end{theorem}

For closed surfaces, every homotopy equivalence is homotopic to a  homeomorphism. So the corresponding result holds trivially in this case.

Homotopy equivalences that are homotopic to homeomorphisms can be characterised as those that preserve the so called \emph{peripheral structure}. However, our result has the advantage that the characterization is in terms of a structure defined without reference to the boundary. Furthermore, the Goldman bracket is the simplest instance of so called \emph{string topology}~\cite{CS1}, which in turn is related to the Floer homology of the cotangent bundle~\cite{AS}.

\section{Preliminaries}

\subsection{The Goldman bracket}
We recall the definition of the Goldman bracket~\cite{Gol} on a connected surface $\Sigma$. Let $\hpi(\Sigma)$ denote the set of free homotopy classes of curves in $\Sigma$. For a closed curve $\alpha$, let $\langle \alpha\rangle\in\hpi(\Sigma)$ denote its homotopy class. 

Given homotopy classes $x$ and $y$ of closed curves, we consider smooth, oriented representatives $\alpha$ and $\beta$ that intersect transversally in double points.  The bracket of $x$ and $y$ is then defined as the sum
\begin{equation}\label{bracket}
[x,y]=\sum_{p\in \alpha\cap\beta} \varepsilon_p\ \langle\alpha*_p\beta\rangle,
\end{equation}
where, for $p\in\alpha\cap\beta$, $\varepsilon_p$  is the sign of the intersection between $\alpha$ and $\beta$ at $p$  and $\alpha*_p\beta$ is the product of the loops $\alpha$ and $\beta$ viewed as based at $p$. 

This expression is well defined by the first part of the following remarkable theorem of Goldman. A completely topological proof of this has been given by Chas~\cite{chas}.

\begin{theorem}[Goldman~\cite{Gol}]\label{gold}
The bracket defined by Equation~\ref{bracket} has the following properties.
\begin{enumerate}
\item Equation~\ref{bracket} gives a well-defined bilinear map $\Z[\hpi(\Sigma)]\times\Z[\hpi(\Sigma)]\to\Z[\hpi(\Sigma)]$, i.e., the right hand side of the equation depends only on the homotopy classes of $\alpha$ and $\beta$. 
\item The bracket is skew-symmetric and satisfies the Jacobi identity.
\item If $\alpha$ is a simple curve and $x=\langle\alpha\rangle$, then, for  $y\in\hpi(\Sigma)$, $[x,y]=0$ if and only if $y=\langle\beta\rangle$ for a closed curve $\beta$ such that $\beta$ is disjoint from $\alpha$.
\end{enumerate}
\end{theorem}

\subsection{Peripheral classes}

For a connected surface $\Sigma$, if $p\in\Sigma$ is a point, $\hpi(\Sigma)$ is the set of conjugacy classes in $\pi_1(\Sigma,p)$. The power operations $\pi_1(\Sigma,p)\to\pi_1(\Sigma,p)$, $\alpha\mapsto\alpha^n$, $n\in\Z$, and the inverse function $\alpha\mapsto\alpha^{-1}$ on $\pi_1(\Sigma,p)$ induce well-defined functions on $\hpi(\Sigma)$. A class $x\in\hpi$ that is not a non-trivial power is called \emph{primitive}.

Suppose henceforth that $\Sigma$ is a compact, connected surface with non-empty boundary. An element of $\hpi$ is said to be \emph{peripheral} if it can be represented by a curve $\alpha\subset \del \Sigma$. 

Assume further that $\Sigma$ has negative Euler characteristic, i.e., $\Sigma$ is not a disc or annulus. Then each component of $\del \Sigma$ corresponds to a pair of primitive peripheral classes (one for each orientation of the boundary curve) which are inverses of each other. Further, the primitive peripheral classes corresponding to different boundary components are different.

\section{Proof of Theorem~\ref{main}}

We now turn to the proof of Theorem~\ref{main}. It is clear from the definition that a homeomorphism commutes with the Goldman bracket. As homotopic maps induce the same function on $\hpi$, it follows that if a homotopy equivalence $f\co \Sigma_1\to\Sigma_2$ is homotopic to a homeomorphism, then $f_*$ commutes with the Goldman bracket.

The rest of the paper is devoted to proving the converse. Assume henceforth that $f\co \Sigma_1\to\Sigma_2$ is a map between connected, compact, oriented surfaces with non-empty boundary so that for all $x,y\in\Z[\hpi(\Sigma_1)]$, we have,
\[
[f_*(x),f_*(y)]=f_*([x,y]).
\]

\begin{lemma}\label{perinul}
For a compact surface $\Sigma$ with boundary, a non-trivial class $x\in\hpi(\Sigma)$ is peripheral if and only if for all $y\in\hpi(\Sigma)$, we have
$[x,y]=0.$
\end{lemma}
\begin{proof}
Assume $x$ is peripheral, i.e., $x=\langle\alpha\rangle$ with $\alpha\subset \del\Sigma$. As any closed curve in $\Sigma$ is freely homotopic to one in the interior of $\Sigma$,  every element $y\in \hpi(\Sigma)$ can be represented by a curve $\beta\subset int(\Sigma)$. Hence $\alpha\cap\beta=\phi$, which implies by Equation~\ref{bracket} that $[x,y]=0$.

Conversely, let $x=\langle\alpha\rangle$ be such that $[x,y]=0$ for all $y\in\hpi(\Sigma)$. If $x$ is not peripheral, it is well known that there is a simple closed curve $\beta$ so that the geometric intersection number of $\alpha$ and $\beta$ is non-zero, i.e., $\alpha$ is not homotopic to a curve that is disjoint from $\beta$. This follows from the fact that there is a pair of simple curves in $\Sigma$ that \emph{fill} $\Sigma$ (i.e., all curves with zero geometric intersection with each of them is peripheral), or alternatively by considering the geodesic representative for $x$ with respect to a hyperbolic metric and using the result that geodesics intersect minimally.

Now, by Goldman's theorem (Theorem~\ref{gold}, part~(3)), as the simple closed curve $\beta$ has non-zero geometric intersection number with $\alpha$ and $y=\langle \beta\rangle$, we have $[x,y]\neq 0$, a contradiction.
 
\end{proof}

Recall that we assume that we have a map $f\co \Sigma_1\to\Sigma_2$ such that $f_*$ commutes with the Goldman bracket. In case $\Sigma_1$ (hence $\Sigma_2$) is a disc or an annulus, it is easy to see that any homotopy equivalence is homotopic to a homeomorphism. We can hence assume henceforth that $\Sigma_1$ and $\Sigma_2$ have negative Euler characteristic.

\begin{lemma}\label{imgbdy}
Suppose $f\co \Sigma_1\to\Sigma_2$ induces a surjection $f_*\co \hpi(\Sigma_1)\to\hpi(\Sigma_2)$. Then if $x\in\hpi(\Sigma_1)$ is a peripheral class, so is $f_*(x)$.
\end{lemma}
\begin{proof}
As $x$ is peripheral, by Lemma~\ref{perinul} $[x,y]=0$ for all $y\in\hpi(\Sigma_1)$. As $f_*$ commutes with the Goldman bracket, it follows that $[f_*(x),f_*(y)]=0$ for all $y\in\hpi(\Sigma_1)$. As $f_*$ is surjective, $[f_*(x),z]=0$ for all $z\in\hpi(\Sigma_2)$. It follows by Lemma~\ref{perinul} that $f_*(x)$ is peripheral.
\end{proof}

\begin{lemma}\label{bdyinv}
Suppose $f\co \Sigma_1\to\Sigma_2$ induces an injection $f_*\co \hpi(\Sigma_1)\to\hpi(\Sigma_2)$. Then, for $x\in\hpi(\Sigma_1)$, if $f_*(x)$ is peripheral then $x$ is peripheral.
\end{lemma}
\begin{proof}
As $f_*(x)$ is peripheral, by Lemma~\ref{perinul} $[f_*(x),f_*(y)]=0$ for all $y\in\hpi(\Sigma_1)$. As $f_*$ is injective and commutes with the Goldman bracket, it follows that $[x,y]=0$ for all $y\in\hpi(\Sigma_1)$. Thus, $x$ is peripheral by Lemma~\ref{perinul}.
\end{proof}

\begin{lemma}\label{bdyhomotbdy}
Suppose $f\co \Sigma_1\to\Sigma_2$ is a homotopy equivalence and $C_1$, $C_2$, \dots $C_n$ are the boundary components of $\Sigma_1$. Then there are boundary components $C'_1$, $C_2'$,\dots, $C'_n$ of $\Sigma_2$ such that $f(C_i)$ is homotopic to $C_i'$ for all $i$, $1\leq i\leq n$. Further the boundary components $C_i'$ are all distinct and $\del\Sigma_2=C_1'\cup C_2'\cup\dots\cup C_n'$. 
\end{lemma}
\begin{proof}
As $f$ is a homotopy equivalence, $f_*\co \hpi(\Sigma_1)\to\hpi(\Sigma_2)$ is both surjective and injective and maps primitive classes to primitive classes. For $1\leq i\leq n$, let $\alpha_i$ be a closed curve parametrizing $C_i$. By Lemma~\ref{imgbdy}, $\alpha_i'=f(\alpha_i)$ represents a primitive peripheral class for $1\leq i\leq n$, and hence represents a boundary component $C'_i$. Furthermore, as $\Sigma_1$ has negative Euler characteristic, for $i\neq j$, $\alpha_i\neq \alpha_j^{\pm 1}$, and hence $\alpha'_i\neq {\alpha'}_j^{\pm 1}$. It follows that the curves $\alpha_i'$ are homotopic to curves representing distinct boundary components $C_i'$, $1\leq i\leq n$. 

Finally, we see that $\del\Sigma_2=C_1'\cup C_2'\cup\dots\cup C_n'$. Namely, if $\gamma$ is a curve parametrizing a boundary component $C''$ of $\del\Sigma_2$, $\langle\gamma\rangle=f_*(x)$ for some class $x$ (as $f_*$ is surjective). The class $x$ must be peripheral  by Lemma~\ref{bdyinv} and primitive as $\gamma$ is primitive. Hence $x$ is represented by either $\alpha_i$ or $\alpha_i^{-1}$ for some $i$, $1\leq i\leq n$. In either case, $C''=C_i'$.
\end{proof}

\begin{lemma}\label{bdy2bdy}
Suppose $f\co \Sigma_1\to\Sigma_2$ is a homotopy equivalence. Then $f$ is homotopic to a map $g\co \Sigma_1\to\Sigma_2$ so that $g(\del \Sigma_1)=\del\Sigma_2$ and $g|_{\del\Sigma_1}\co  \del\Sigma_1\to\del\Sigma_2$ is a homeomorphism.
\end{lemma}
\begin{proof}
By Lemma~\ref{bdyhomotbdy}, the restriction of $f$ to $\del\Sigma_1$ is homotopic in $\Sigma_2$ to a homeomorphism $\varphi\co \del\Sigma_1\to\del\Sigma_2$. Let $h\co \del\Sigma_1\times [0,1]\to\del\Sigma_2$ be such a homotopy, i.e., $h$ is a map so that $h(\cdot,0)=f(\cdot)$ and $h(\cdot,1)=\varphi(\cdot)$. 

Let $\Sigma_1'=\Sigma_1\coprod_{\del\Sigma_1=(\del\Sigma_1\times\{0\})} \del\Sigma_1\times [0,1]$ be the space obtained from the disjoint union of $\Sigma_1$ and $\del\Sigma_1\times [0,1]$ by identifying points $x\in\del\Sigma_1\subset \Sigma_1$ with the corresponding points $(x,0)\in \del\Sigma_1\times [0,1]$. Define a map $g'\co \Sigma_1'\to\Sigma_2$ by $g'(x)=f(x)$ for $x\in \Sigma_1$ and $g'(x,t)=h(x,t)$ for $(x,t)\in \del\Sigma_1\times [0,1]$. 

Observe that $\del\Sigma_1'=\del\Sigma_1\times\{1\}$. The inclusion map $\Sigma_1\to\Sigma_1'$ is homotopic to a homeomorphism $\psi$ so that, for $x\in\del\Sigma_1$, $\psi(x)=(x,1)$. Let $H'\co \Sigma_1\to\Sigma_1'$ be such a homotopy, i.e., $H'$ is a map so that $H'(\cdot,0)$ is the inclusion map and $H'(\cdot, 1)=\psi(\cdot)$. Let $g\co \Sigma_1\to\Sigma_2$ be given by
$$g(x)=g'(H'(x,1))=g'(\psi(x)).$$
Then, for $x\in\del\Sigma_1$, $g(x)=g'(\psi(x))=g'((x,1))=h(x,1)=\varphi(x)$, i.e. $g\vert_{\del \Sigma_1}=\varphi$. Hence $g$ is a map from $\Sigma_1\to\Sigma_2$ so that $g(\del \Sigma_1)=\del\Sigma_2$ and $g|_{\del\Sigma_1}\co  \del\Sigma_1\to\del\Sigma_2$ is a homeomorphism. Finally, for $x\in\Sigma_1$, $g'(H'(x,0))=g'(x)=f(x)$. Hence, 
a homotopy from $f$ to $g$ is given by
$$H(x,t)=g'(H(x,t)).$$
\end{proof}

By a theorem of Nielsen, $g$, and hence $f$, is homotopic to a homeomorphism  (see Lemma 1.4.3 of~\cite{Wa} -- as $g\colon(\Sigma_1,\del\Sigma_1)\to (\Sigma_2,\del\Sigma_2)$ induces an isomorphism of fundamental groups, it follows that $g$ is homotopic to a homeomorphism). This completes the proof of Theorem~\ref{main}.\qed

\end{document}